\documentclass[submission]{FPSAC2026}

\allowdisplaybreaks

\newtheorem{tm}{Theorem}
\newtheorem{df}{Definition}
\newtheorem{lm}{Lemma}
\newtheorem{pp}{Proposition}
\newtheorem{cor}{Corollary}

\newtheorem{examp}{Example}

\usepackage{tikz}
\usepackage{tikz-cd}

\title[Generalized Interval Polynomial]{The Generalized Interval Polynomial of a Poset}

\author[I. George \and K. Yeats]{Ian George\addressmark{1} \and Karen Yeats\thanks{\href{mailto:kayeats@uwaterloo.ca}{kayeats@uwaterloo.ca}. Karen Yeats is partially supported by an NSERC Discovery grant and the Canada Research Chairs program.}\addressmark{1}}

\address{\addressmark{1}Department of Combinatorics and Optimization, University of Waterloo, Waterloo ON}

\received{\today}

\abstract{For any finite poset we define a generating polynomial counting upsets, downsets, and their intersection.  We investigate the behaviour of this polynomial with respect to poset operations, show that it distinguishes series-parallel posets, and comment on connections to the causal set approach to quantum gravity.}


\begin{document}

\maketitle

\section{Introduction}

In this paper we introduce and investigate a generating polynomial for upsets, downsets, and their intersection, of a finite poset.  Enumerating upsets/downsets and intervals in posets is a problem that has garnered attention in the recent literature.  For many nice families of posets, such enumeration involves rich combinatorics, e.g. \cite{bousquet2023intervals, azam2023generating, baril2025enumeration, chapoton2006nombre, preville2017enumeration, melou18number, baril2024intervals}.  The intersection of an arbitrary upset and downset is a generalization of an interval, that being the case where the upset and downset are each generated by a single element.

Our polynomial is a polynomial with five variables which records 
the upset and downset sizes, the size of their intersection, and the sizes of the antichains that generate them.  The pairs of upsets and downsets we consider satisfy a containment condition, which gives a simple product formula for the polynomial on disjoint union of posets.

This polynomial is also motivated from the causal set approach to quantum gravity wherein the enumeration of intervals in a poset connects to spacetime dimension.

\section{The Polynomial}

\subsection{Definitions}

All posets considered here will be finite, so we will simply say poset instead of finite poset.  Let $P$ be a poset with partial order $\leq$.  A set $U \subseteq P$ ($D \subseteq P$) is \emph{upward-closed} (\emph{downward-closed}), or is an \emph{upset} (\emph{downset}), if for each $u \in U$ ($u \in D$), if $v \in P$ and $u \leq v$ ($v \leq u$), then $v \in U$ ($v \in D$).  For a set $S \subseteq P$ we say the \emph{upset generated by $S$} is $V(S) = \{v \in P \; : \; u \leq v \; {\rm for \; some} \; u \in S\}$, and similarly the \emph{downset generated by $S$} is $\Lambda(S) = \{u \in P \;:\; u \leq v \; {\rm for \; some} \; v \in S\}$.  Indeed, $V(S)$ ($\Lambda(S)$) is an upset (downset) of $P$.  When $S = \{u\}$, we write $V(u)$ ($\Lambda(u)$) instead of $V(\{u\})$ ($\Lambda(\{u\})$) for the \emph{principal upset (downset) generated by $u$}.  Given an upset $U \subseteq P$ (downset $D \subseteq P$), there is a unique antichain $A = A(U)$ ($A = A(D)$) such that $U = V(A)$ ($D =\Lambda(A)$).

For a pair consisting of an upset $U$ and downset $D$ of $P$, we say that $U$ is \emph{beneath} $D$ or that $D$ is \emph{above} $U$, denoted $U \preceq D$, if $A(U) \subseteq D$ and $A(D) \subseteq U$.  In this case, we say the set $U \cap D$ is a \emph{generalized interval} of $P$.  If $A(U) = \{u\}$ and $A(D) = \{v\}$, then $U \cap D$ is simply an \emph{interval}, denoted by $[u, v]$.

Now we can define our object of interest.

\begin{df}
    The \emph{generalized interval polynomial} $\Phi(P ; s, t, x, y, z)$ of a poset $P$ is

    \[ \Phi(P ; s, t, x, y, z) = \sum_{U \preceq D} s^{|A(U)|} t^{|A(D)|} x^{|D|} y^{|U|} z^{|U \cap D|} \]

    \noindent where the sum runs over all upsets $U \subseteq P$ and downsets $D \subseteq P$ such that $U \preceq D$.
\end{df}

\noindent Note that we allow upsets and downsets to be empty, so every such polynomial will have a constant term of $1$ corresponding to $U = D = \emptyset$.  Also note that if either the upset or downset is empty, the other is necessarily empty due to the $\preceq$ condition, hence every non-constant term in a generalized interval polynomial is divisible by $stxyz$. 

\begin{examp}
    Consider the poset $P$ on the left in Figure \ref{fig:combined_examples} with elements $\{a, b, c, d\}$ and non-trivial relations $a \leq b \leq d$ and $a \leq c \leq d$.  The generalized interval polynomial of $P$ is

    \begin{align*}
        \Phi(P ; s, t, x, y, z) = 1+ stxyz(x^3y^3z^3 & + sx^3y^2z^2 +  
        tx^2y^3z^2 + stx^2y^2z \\
        & \quad + 2x^3yz + 2xy^3z + x^3 + y^3 + 2xy)
    \end{align*}
    
    \begin{figure}
        \centering
        \begin{tikzpicture}[scale=0.9]
            \node (a) at (0, 0) {};
            \node [circle, fill, inner sep=3pt, label={[xshift=-0.0cm, yshift=-0.8cm]$a$}] at (0, 0) {};
            \node (b) at (-0.5, 0.8666) {};
            \node [circle, fill, inner sep=3pt, label={[xshift=-0.34cm, yshift=-0.4cm]$b$}] at (-0.5, 0.8666) {};
            \node (c) at (0.5, 0.8666) {};
            \node [circle, fill, inner sep=3pt, label={[xshift=0.34cm, yshift=-0.4cm]$c$}] at (0.5, 0.8666) {};
            \node (d) at (0, 1.732) {};
            \node [circle, fill, inner sep=3pt, label={[xshift=-0.0cm, yshift=0.0cm]$d$}] at (0, 1.732) {};
            
            \draw (a) -- (b) -- (d);
            \draw (a) -- (c) -- (d);
        \end{tikzpicture}
        \hspace{1cm}
    \begin{tikzpicture}[scale=0.6]
        \node (a) at (0, 0) {};
        \node [circle, fill, inner sep=3pt] at (0, 0) {};
        \node (b) at (2, 0) {};
        \node [circle, fill, inner sep=3pt] at (2, 0) {};
        \node (c) at (0, 2) {};
        \node [circle, fill, inner sep=3pt] at (0, 2) {};
        \node (d) at (2, 2) {};
        \node [circle, fill, inner sep=3pt] at (2, 2) {};
        \node (e) at (0, 4) {};
        \node [circle, fill, inner sep=3pt] at (0, 4) {};
        \node (f) at (2, 4) {};
        \node [circle, fill, inner sep=3pt] at (2, 4) {};

        \draw (a) -- (c) -- (b) -- (d);
        \draw (c) -- (e) -- (d) -- (f);


        \node (1) at (4, 0) {};
        \node [circle, fill, inner sep=3pt] at (4, 0) {};
        \node (2) at (6, 0) {};
        \node [circle, fill, inner sep=3pt] at (6, 0) {};
        \node (3) at (4, 2) {};
        \node [circle, fill, inner sep=3pt] at (4, 2) {};
        \node (4) at (6, 2) {};
        \node [circle, fill, inner sep=3pt] at (6, 2) {};
        \node (5) at (5, 3) {};
        \node [circle, fill, inner sep=3pt] at (5, 3) {};
        \node (6) at (5, 4) {};
        \node [circle, fill, inner sep=3pt] at (5, 4) {};

        \draw (1) -- (3) -- (2) -- (4);
        \draw (4) -- (5) -- (6);
        \draw (1) -- (6);
    \end{tikzpicture}
        \hspace{1cm}
        \begin{tikzpicture}[scale=0.5]
        \node (a) at (0, 0) {};
        \node [circle, fill, inner sep=3pt] at (0, 0) {};
        \node (b) at (2, 0) {};
        \node [circle, fill, inner sep=3pt] at (2, 0) {};
        \node (c) at (0, 2) {};
        \node [circle, fill, inner sep=3pt] at (0, 2) {};
        \node (d) at (2, 2) {};
        \node [circle, fill, inner sep=3pt] at (2, 2) {};
        \draw (a) -- (c) -- (b) -- (d);
    \end{tikzpicture}
        \caption{Some example posets of interest.}
        \label{fig:combined_examples}
    \end{figure}
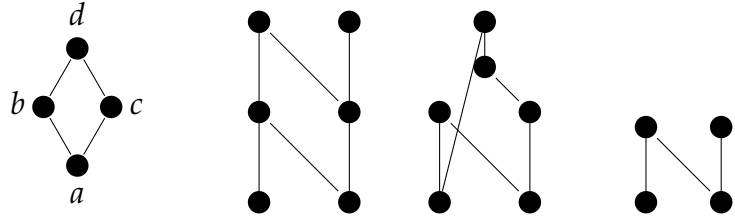
\end{examp}

\subsection{Polynomial Specifications/Combinatorial Information}

We can extract combinatorial information about $P$ by extracting coefficients and taking certain evaluations of $\Phi(P ; s, t, x, y, z)$, as summarized in Table~\ref{tab:info_from_poly}.  
Note that coefficient extraction is equivalent to taking a partial derivative and applying an appropriate scaling, so the results below could equivalently be phrased in that language.  

\begin{table}[h]
    \centering
    \begin{tabular}{|c||c|}
    \hline
    Poset information/property & How to obtain it \\
    \hline
    Number of elements & \begin{tabular}{@{}c@{}}$[stz]\Phi(P ; s, t, 1, 1, z)$, or maximum $n$ \\ such that $[z^n]\Phi(P ; s, t, x, y, z) \neq 0$\end{tabular} \\
    \hline
    Number of cover relations & $[stz^2]\Phi(P ; s, t, 1, 1, z)$ \\
    \hline
    Number of relations & $[s t] \Phi(P ; s, t, 1, 1, 1)$ \\
    \hline
    Number of minimal elements & $m$ such that $[s^m y^{|P|}]\Phi(P ; s, t, x, y, z) \neq 0$ \\
    \hline
    Number of maximal elements & $m$ such that $[t^m x^{|P|}]\Phi(P ; s, t, x, y, z) \neq 0$ \\
    \hline
    Number of antichains of size $k$ & $[s^k t^k z^k]\Phi(P ; s, t, 1, 1, z)$ \\
    \hline
    Poset width & Maximum $k$ such that $[s^kt^kz^k]\Phi(P ; s, t, x, y, z) \neq 0$ \\
    \hline
    \end{tabular}
    \caption{Extracting information about $P$ from the polynomial.}
    \label{tab:info_from_poly}
\end{table}

In addition to the above information, we can specialize $\Phi(P; s, t, x, y, z)$ to other generating polynomials.  In Table \ref{tab:other_gens_from_poly} we summarize some other generating polynomials that can be obtained from $\Phi$.  Note that in any term of the generalized interval polynomial the exponent of $x$ and $y$ is greater than or equal to the exponent of $z$, and the exponent of $z$ is greater than or equal to the exponent of $s$ and $t$.  Thus for any of the substitutions involving $w^{-1}$ below, the result is still a polynomial and so we may set $w=0$.

\begin{table}[h]
    \centering
    \begin{tabular}{|c||c||c|}
    \hline
    Generating polynomial for sizes of... & How to obtain it & Found in \\
    \hline
    Generalized intervals & $\Phi(P ; 1, 1, 1, 1, z)$ & \\
    \hline
    Upsets, downsets, generalized intervals & $\Phi(P ; 1, 1, x, y, z)$ & \\
    \hline
    Upsets & $\Phi(P ; 1, 1, 1, yw, 1/w)|_{w=0}$ & \\
    \hline
    Downsets & $\Phi(P ; 1, 1, xw, 1, 1/w)|_{w=0}$ & \cite{ding2019antichain}, \cite{proctor1984bruhat} \\
    \hline
    Principal upsets & $[s]\Phi(P ; s, 1, 1, yw, 1/w)|_{w=0}$ & \cite{sorkin_comm} \\
    \hline
    Principal downsets & $[t]\Phi(P ; 1, t, 1, yw, 1/w)|_{w=0}$ & \cite{sorkin_comm} \\
    \hline
    Principal upsets and downsets & $[z] \Phi(P ; 1, 1, x, y, z)$ & \cite{sorkin_comm} \\
    \hline
    Intervals & $[st] \Phi(P ; s, t, 1, 1, z)$ & \\
    \hline
    Principal upsets, downsets, intervals & $[st] \Phi(P ; s, t, x, y, z)$ & \\
    \hline
    Antichains & $\Phi(P ; 1/w, 1, 1, 1, zw)|_{w=0}$ & \cite{ding2019antichain} \\
    \hline
    \end{tabular}
    \caption{Other generating polynomials obtained from $\Phi(P ; s, t, x, y, z)$.}
    \label{tab:other_gens_from_poly}
\end{table}


\subsection{Distinguishability}

We saw in the previous section some of the combinatorial information that is captured in the generalized interval polynomial.  We now ask the natural question: exactly how much can we determine about a poset $P$ just from its generalized interval polynomial?  
Computations in SageMath \cite{sagemath} reveal that $P$ is not determined (up to isomorphism) by its generalized interval polynomial.  
At six elements, the middle two posets of Figure~\ref{fig:combined_examples} are the smallest two posets that are indistinguishable by their generalized interval polynomials.  For seven elements there are 16 pairs of posets which are not distinguished, while for eight elements there 249 pairs.

\subsection{Poset Operations}

As $\Phi(P ; s, t, x, y, z)$ is a generating polynomial encoding combinatorial information about $P$, it is of interest to know how it behaves with poset operations.  Given a poset $P$, its \emph{dual} $P^*$ is the poset with the same underlying set but with relations $u \leq_{P^*} v$ if and only if $v \leq_P u$.  It is immediate from this definition that $S$ is an upset of $P$ if and only if $S$ is a downset of $P^*$, and $A$ is an antichain of $P$ if and only if it is an antichain of $P^*$.  Thus we find that $\Phi(P^* ; s, t, x, y, z) = \Phi(P ; t, s, y, x, z)$.

Given two posets $P, Q$ on disjoint underlying sets, their \emph{disjoint union} $P+Q$ is the poset with underlying set $P \cup Q$ and relations $u \leq_{P + Q} v$ if and only if $u, v \in P$ and $u \leq_P v$ or $u, v \in Q$ and $u \leq_Q v$.

\begin{pp}
    $\Phi(P+Q ; s, t, x, y, z) = \Phi(P ; s, t, x, y, z) \Phi(Q ; s, t, x, y, z)$.
    \label{prop:union}
\end{pp}

\begin{proof}
    Let $U$ be an upset of $P+Q$ and $D$ a downset of $P+Q$ such that $U \preceq D$.  Then we can write $U = U_P \cup U_Q$, where $U_P$ is an upset of $P$ and $U_Q$ an upset of $Q$.  Similarly, we have $D = D_P \cup D_Q$.  These decompositions are unique, since $U_P = U \cap P$ and $U_Q = U \cap Q$ (and similarly for the downsets).  Since $U \preceq D$, it follows that $U_P \preceq D_P$ and $U_Q \preceq D_Q$.  Furthermore, $|A(U)| = |A(U_P)| + |A(U_Q)|$, $|U| = |U_P| + |U_Q|$, with analogous expressions holding for $D$, and $|U \cap D| = |U_P \cap D_P| + |U_Q \cap D_Q|$.  Since each $U_P \cup U_Q$, where $U_P, U_Q$ are upsets of $P, Q$, respectively, is an upset of $P + Q$ (and similarly for downsets), the upsets of $P+Q$ are exactly those of the form $U_P \cup U_Q$ (and similarly for downsets).  The result then follows from polynomial multiplication.
\end{proof}

The \emph{ordinal sum}, $P \oplus Q$, of $P$ and $Q$ also has underlying set $P \cup Q$, but has the relations of $P+Q$ as well as the relations $u \leq_{P \oplus Q} v$ if $u \in P$, $v \in Q$.  The cover relations of $P \oplus Q$ are precisely the cover relations of $P$ and $Q$ along with the relations $u \leq v$ such that $u$ is a maximal element of $P$ and $v$ a minimal element of $Q$.

\begin{pp}
    \begin{align*}
        \Phi(P \oplus Q ; s, t, x, y, z) & = y^{|Q|}(\Phi(P; s, t, x, y, z)-1) + x^{|P|}(\Phi(Q; s, t, x, y, z)-1) \\
        & \quad + x^{|P|} y^{|Q|} \Phi(P ; s, 1, 1, yw, z/w)|_{w=0} \Phi(Q ; 1, t, xw, 1, z/w)|_{w=0} + 1
    \end{align*}
\end{pp}

\begin{proof}
    For an upset/downset pair $U \preceq D$ of $P \oplus Q$, we consider four cases.  If $U = D = \emptyset$, then $U \preceq D$ contributes the term $1$ to the polynomial.  For the remaining cases, we assume $U, D \neq \emptyset$.  If $A(U), A(D) \subseteq P$, then we get a term that is exactly the term from $\Phi(P ; s, t, x, y, z)$ corresponding to $(U \cap P) \preceq D$ multiplied by $y^{|Q|}$ since $|U| = |U \cap P|+|Q|$.  Furthermore, every non-constant term of $\Phi(P ; s, t, x, y, z)$ contributes a term to $\Phi(P \oplus Q ; s, t, x, y, z)$ in this way.  Similarly, when $A(U), A(D) \subseteq Q$, $|D| = |D \cap Q| + |P|$ we get a term that is the term from $\Phi(Q ; s, t, x, y, z)$ corresponding to $U \preceq (D \cap Q)$ multiplied by $x^{|P|}$.  The final case is when $U \preceq D$ is such that $A(U) \subseteq P$ and $A(D) \subseteq Q$.  Here, $|U| = |U \cap P| + |Q|$, $|D| = |D \cap Q| + |P|$, and $|U \cap D| = |U \cap P| + |D \cap Q|$.  Thus we get a term

    \[ x^{|P|} y^{|Q|}( s^{|A(U)|} y^{|U \cap P|} z^{|U \cap P|}) ( t^{|A(D)|} x^{|D \cap Q|} z^{|D \cap Q|}) \]

    \noindent which is precisely the product of the term from $y^{|Q|} \Phi(P ; s, 1, 1, yw, z/w)|_{w=0}$ corresponding to $V(A(U)) \cap P \preceq \Lambda(A(U))$ and the term from $x^{|P|}\Phi(Q ; 1, t, xw, 1, z/w)|_{w=0}$ corresponding to $V(A(D)) \preceq \Lambda(A(D)) \cap Q$.  These are the only cases we need to consider since no antichain of $P \oplus Q$ has non-empty intersection with both $P$ and $Q$.  All terms on the right hand side appear in exactly one of these four cases, so the result follows.
\end{proof}

The Cartesian product $P \times Q$ of posets $P, Q$ is the poset with underlying set $P \times Q$ (the ordinary Cartesian product of sets) and relations $(u_P, u_Q) \leq (v_P, v_Q)$ if and only if $u_P \leq v_P$ and $u_Q \leq v_Q$.  This is a commutative operation, i.e. $P \times Q \cong Q \times P$.  There appears to be no simple formula for $\Phi(P \times Q ; s, t, x, y, z)$ in terms of $\Phi(P ; s, t, x, y, z)$ and $\Phi(Q ; s, t, x, y, z)$, but we may write a formula for the simpler polynomial $\phi(P \times Q ; x, y, z) = [st]\Phi(P \times Q ; s, t, x, y, z)$.  Recall that this is the generating polynomial of (non-empty) interval sizes which also keeps track of the sizes of the principal upset and downset giving the interval.

\begin{pp}
    Let $\phi(P ; x, y, z) = \sum p_{i, j, k} x^i y^j z^k$ and $\phi(Q ; x, y, z) = \sum q_{i, j, k} x^i y^j z^k$.  Then
    
    \[ \phi(P \times Q ; x, y, z) = \sum q_{i, j, k} \phi(P ; x^i, y^j, z^k) = \sum p_{i, j, k} \phi(Q ; x^i, y^j, z^k) \]
\end{pp}

\begin{proof}
    First observe that an interval $[(u_P, u_Q), (v_P, v_Q)]$ in $P \times Q$ where $u_P, v_P \in P$ and $u_Q, v_Q \in Q$ has the form $[(u_P, u_Q), (v_P, v_Q)] = [u_P, v_P] \times [u_Q, v_Q]$.  For an upset $V((u_P, u_Q))$ in $P \times Q$ we have $V(u_P, u_Q) = V(u_P) \times V(u_Q)$.  Similarly for a downset, $\Lambda((v_P, v_Q)) = \Lambda(v_P) \times \Lambda(v_Q)$.  Now we just perform the calculation:

    \begin{align*}
        \phi(P \times Q ; x, y, z) & = \sum_{(u_P, u_Q) \leq (v_P, v_Q)} x^{|\Lambda((v_P, v_Q))|} y^{|V((u_P, u_Q))|} z^{|[(u_P, u_Q), (v_P, v_Q)]|} \\
        & = \sum_{u_Q \leq v_Q} \sum_{u_P \leq v_P} x^{|\Lambda(v_P)| \cdot |\Lambda(v_Q)|} y^{|V(u_P)| \cdot |V(u_Q)|} z^{|[u_P, v_P]| \cdot |[u_Q, v_Q]|} \\
        & = \sum_{u_Q \leq v_Q} \sum_{u_P \leq v_P} (x^{|\Lambda(v_Q)|})^{|\Lambda(v_P)|} (y^{|V(u_Q)|})^{|V(u_P)|} (z^{|[u_Q, v_Q]|})^{|[u_P, v_P]|} \\
        & = \sum_{u_Q \leq v_Q} \phi(P ; x^{|\Lambda(v_Q)|}, y^{|V(u_Q)|}, z^{|[u_Q, v_Q]|}) \\
        & = \sum q_{i, j, k} \phi(P ; x^i, y^j, z^k)
    \end{align*}

    \noindent The other equality follows directly from swapping sums.
\end{proof}

\subsection{Families of Posets}

Here we find explicit formulas for the generalized interval polynomials of some nice families of posets.  

\begin{examp}
    If $P$ consists of a single element, then $\Phi(P; s, t, x, y, z) = 1+stxyz$.  Thus if $P$ is an antichain of $n$ elements, then $\Phi(P; s, t, x, y, z) = (1+stxyz)^n$.
\end{examp}

\begin{examp}   
    If $P = [n]$ is a chain consisting of $n$ elements, then

    \begin{align*}
        \Phi(P; s, t, x, y, z) & = 1 + st \sum_{k=1}^n z^k \sum_{i=0}^{n-k} x^{k+i}y^{n-i} = 1 + st\sum_{k=1}^n z^k \frac{x^{n+1}y^k - x^k y^{n+1}}{x-y} \\
        & = 1 + \frac{st}{x-y} \left( x^{n+1} \frac{1-(yz)^{n+1}}{1-yz} - y^{n+1}\frac{1-(xz)^{n+1}}{1-xz} \right)
    \end{align*}
\end{examp}

\begin{examp}
    The \emph{Boolean lattice} $B_n$ is the poset of subsets of $[n]$ ordered under inclusion.  
    We do not have a nice expression for $\Phi(B_n ; s, t, x, y, z)$.  However, we do have, from a direct computation,
    
    \[ \phi(B_n ; x, y, z) = \sum_{\ell=0}^n \sum_{k=0}^\ell \binom{n}{\ell} \binom{\ell}{k} x^{2^\ell} y^{2^{n-k}} z^{2^{\ell-k}} \]
\end{examp}

\begin{examp}
    Consider the product of chains $P = [n_1] \times [n_2] \times \cdots \times [n_k]$.  We do not have an expression for $\Phi(P ; s, t, x, y, z)$, but for the ordinary interval generating polynomial we have:
    \[ \phi(P ; x, y, z) = \sum_{1 \leq i_1 \leq j_1 \leq n_1} \cdots \sum_{1 \leq i_k \leq j_k \leq n_k} x^{\prod_r j_r} y^{\prod_r (n_r - i_r + 1)} z^{\prod_r (j_r - i_r + 1)} \]
    \noindent where in each exponent the product ranges  over $r \in [k]$.

    Consider the case of $k=2$, so $P = [n_1] \times [n_2]$ is a product of two chains.  Antichains of $P$ are in correspondence with other combinatorial objects such as alignments of strings, words in a three-letter alphabet, and endpoints of subsequences in walks \cite{bouyssou2024maximal}.  Thus $\Phi(P ; s, t, x, y, z)$ encodes enumerative information of these other objects.
\end{examp}


\section{Series-Parallel Posets}

In this section we focus on generalized interval polynomials of the well known Series-Parallel posets.  These posets are built with the $+$ and $\oplus$ operations, for which we know how to compute the generalized interval polynomial.  A generating series for the number of these posets was found in \cite{stanley1974enumeration}, and the asymptotics were computed in \cite{bayoumi1989asymptotic}.

\begin{df}
    A \emph{Series-Parallel (SP) poset} is a poset that is constructed inductively by the operations of $+$ and $\oplus$, starting with single elements.  Equivalently, a poset is SP if and only if it is $N$-free, that is it does not contain a subposet isomorphic to the $N$ poset which is illustrated on the right in Figure~\ref{fig:combined_examples}.
\end{df}


SP posets are uniquely determined by their \emph{composition tree} (see, for example, \cite{schrader1997setup}), which is a rooted tree that describes the tree structure of $+$ and $\oplus$ operations used to construct the poset starting from single elements.  Example \ref{ex:comp_tree} illustrates how such a tree determines an SP poset.

\begin{examp}
    Consider the SP poset $P$ shown in the left of Figure \ref{fig:ex_of_sp_tree}.  Since $P$ is a disjoint union of two posets that are themselves not a disjoint union, the root $r$ of the composition tree $T$ in the right of Figure \ref{fig:ex_of_sp_tree} is a vertex labelled $+$ with two neighbours.  The two-element component of $P$ is the ordinal sum of the single-element posets with vertices $a$ and $b$, so the neighbour of $r$ corresponding to this component is a vertex labelled $\oplus$ which has two leaves, labelled $a$ and $b$.  Similarly, the other component of $P$ is the ordinal sum of the antichain $\{c, d\}$ with $\{e\}$.  Note that the plane structure of the tree at each vertex corresponding to a $\oplus$ operation matters, since $\{c, d\} \oplus \{e\} \neq \{e\} \oplus \{c, d\}$.

    \begin{figure}[h]
        \centering
        \begin{tikzpicture}[scale=0.8]
            \node (a) at (0, 0) {};
            \node [circle, fill, inner sep=3pt, label={[xshift=-0.3cm, yshift=-0.35cm]$a$}] at (0, 0) {};
            \node (b) at (0, 2) {};
            \node [circle, fill, inner sep=3pt, label={[xshift=-0.3cm, yshift=-0.35cm]$b$}] at (0, 2) {};
            \node (c) at (1, 0) {};
            \node [circle, fill, inner sep=3pt, label={[xshift=-0.3cm, yshift=-0.35cm]$c$}] at (1, 0) {};
            \node (d) at (3, 0) {};
            \node [circle, fill, inner sep=3pt, label={[xshift=-0.3cm, yshift=-0.35cm]$d$}] at (3, 0) {};
            \node (e) at (2, 2) {};
            \node [circle, fill, inner sep=3pt, label={[xshift=-0.3cm, yshift=-0.35cm]$e$}] at (2, 2){};
            
            \draw (a) -- (b);
            \draw (c) -- (e) -- (d);

            \node (1) at (7, 0) {};
            \node [circle, inner sep=3pt] at (7, 0) {$+$};
            \node (2) at (6, 0.666666) {};
            \node [circle, inner sep=3pt] at (6, 0.666666) {$\oplus$};
            \node (3) at (8, 0.666666) {};
            \node [circle, inner sep=3pt] at (8, 0.666666) {$\oplus$};
            \node (4) at (5.333333, 1.333333) {};
            \node [circle, inner sep=3pt] at (5.333333, 1.333333) {$a$};
            \node (5) at (6.666666, 1.333333) {};
            \node [circle, inner sep=3pt] at (6.666666, 1.333333) {$b$};
            \node (6) at (7.333333, 1.333333) {};
            \node [circle, inner sep=3pt] at (7.333333, 1.333333) {$+$};
            \node (7) at (8.666666, 1.333333) {};
            \node [circle, inner sep=3pt] at (8.666666, 1.333333) {$e$};
            \node (8) at (6.666666, 2) {};
            \node [circle, inner sep=3pt] at (6.666666, 2) {$c$};
            \node (9) at (8, 2) {};
            \node [circle, inner sep=3pt] at (8, 2) {$d$};

            \draw (2) -- (1) -- (3);
            \draw (4) -- (2) -- (5);
            \draw (6) -- (3) -- (7);
            \draw (8) -- (6) -- (9);

        \end{tikzpicture}
        \caption{An SP poset and its composition tree.}
        \label{fig:ex_of_sp_tree}
    \end{figure}
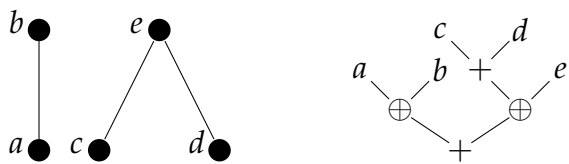
    \label{ex:comp_tree}
\end{examp}

While the generalized interval polynomial does not distinguish all posets, 
in this section we prove that it does distinguish the class of SP posets. Distinguishing SP posets has been investigated previously in the literature: in \cite{gordon1996series} it was shown that a strict subclass of SP posets are distinguishable via their \emph{Tutte polynomial}, while it was left undetermined whether the Tutte polynomial distinguishes all SP posets.  Simple computation reveals that the ordinary interval polynomial $\phi$ fails to distinguish SP posets, see Figure~\ref{fig:sp_pos_ord_poly_nondis}.  

\begin{figure}[h]
    \centering
    \begin{tikzpicture}[scale=0.8]
        \node (a) at (0, 0) {};
        \node [circle, fill, inner sep=3pt] at (0, 0) {};
        \node (b) at (-1, 1) {};
        \node [circle, fill, inner sep=3pt] at (-1, 1) {};
        \node (c) at (-0.333333, 1) {};
        \node [circle, fill, inner sep=3pt] at (-0.333333, 1) {};
        \node (d) at (0.333333, 1) {};
        \node [circle, fill, inner sep=3pt] at (0.333333, 1) {};
        \node (e) at (1, 1) {};
        \node [circle, fill, inner sep=3pt] at (1, 1) {};

        \node (f) at (2, 0) {};
        \node [circle, fill, inner sep=3pt] at (2, 0) {};
        \node (g) at (1.666666, 1) {};
        \node [circle, fill, inner sep=3pt] at (1.666666, 1) {};
        \node (h) at (2.333333, 1) {};
        \node [circle, fill, inner sep=3pt] at (2.333333, 1) {};
        \node (i) at (1.666666, 2) {};
        \node [circle, fill, inner sep=3pt] at (1.666666, 2) {};
        \node (j) at (2.333333, 2) {};
        \node [circle, fill, inner sep=3pt] at (2.333333, 2) {};

        \draw (b) -- (a);
        \draw (c) -- (a);
        \draw (d) -- (a);
        \draw (e) -- (a);
        \draw (i) -- (g) -- (f) -- (h) -- (j);

        \node (1a) at (4.666666, 0) {};
        \node [circle, fill, inner sep=3pt] at (4.666666, 0) {};
        \node (1b) at (4, 1) {};
        \node [circle, fill, inner sep=3pt] at (4, 1) {};
        \node (1c) at (4.666666, 1) {};
        \node [circle, fill, inner sep=3pt] at (4.666666, 1) {};
        \node (1d) at (5.333333, 1) {};
        \node [circle, fill, inner sep=3pt] at (5.333333, 1) {};
        \node (1e) at (6, 1) {};
        \node [circle, fill, inner sep=3pt] at (6, 1) {};

        \node (1f) at (6.666666, 0) {};
        \node [circle, fill, inner sep=3pt] at (6.666666, 0) {};
        \node (1g) at (6.666666, 1) {};
        \node [circle, fill, inner sep=3pt] at (6.666666, 1) {};
        \node (1h) at (7.333333, 1) {};
        \node [circle, fill, inner sep=3pt] at (7.333333, 1) {};
        \node (1i) at (4.666666, 2) {};
        \node [circle, fill, inner sep=3pt] at (4.666666, 2) {};
        \node (1j) at (6.666666, 2) {};
        \node [circle, fill, inner sep=3pt] at (6.666666, 2) {};

        \draw (1b) -- (1a);
        \draw (1i) -- (1c) -- (1a) -- (1d);
        \draw (1e) -- (1f);
        \draw (1j) -- (1g) -- (1f) -- (1h);
    \end{tikzpicture}
    \caption{Two SP posets that are not distinguished by the ordinary interval polynomial.}
    \label{fig:sp_pos_ord_poly_nondis}
\end{figure}
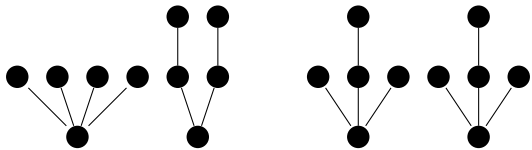

Our proof that the generalized interval polynomial distinguishes SP posets will go as follows: first, we show that we can recover the sizes of each summand in an ordinal sum of posets from the generalized interval polynomial.  Using this, we recover the generalized interval polynomial of each summand.  We then prove that the generalized interval polynomial of an ordinal sum does not non-trivially factor.  So, given the generalized interval polynomial we can determine if an SP poset is a disjoint union or an ordinal sum and the polynomials of each factor of these operations.  Inductively, in this way we construct the composition tree of the poset, hence uniquely determining the poset.

\begin{lm}
    Let $P = P_1 \oplus \cdots \oplus P_r$ be an ordinal sum of posets where each $P_i$ cannot be written as an ordinal sum of strictly smaller posets.  Then we can recover $|P_1|, \ldots, |P_r|$ from $\Phi(P ; s, t, x, y, z)$.
    \label{lm:ord_sum_sizes}
\end{lm}

\begin{proof}
    We know from Table \ref{tab:info_from_poly} that $|P| = [s t z]\Phi(P ; s, t, 1, 1, z)$, and that the cover relations of $P$ exactly correspond to the coefficient $[s t z^2] \Phi(P ; s, t, x, y, z)$.  Consider the terms of $\Phi(P ; s, t, x, y, z)$ of the form $c s t x^b y^a z^2$ where $c > 0$, $a+b = |P|+2$.  These terms correspond to cover relations $u < v$ such that for every $w \in P$, either $u \leq w$ or $w \leq v$.  A cover relation such that $u$ is a maximal element of $P_\ell$ and $v$ a minimal element of $P_{\ell+1}$ for some $\ell < r$ corresponds to a term of this form, but not all terms of this form correspond to these cover relations.  The procedure below will identify those terms that do.
    
    Define $b = (b_1, \ldots, b_m)$ to be the sequence of integers such that for each $k$, the term $stx^{b_k+1}y^{|P|-b_k+1}z^2$ has non-zero coefficient $c_k$, where $b_k < b_{k+1}$.  The rest of the proof will proceed as an algorithm where we will determine the size of each ordinal summand.
    
    \noindent \textbf{Step 0.} Set $k, \ell=1$.
    
    \noindent \textbf{Step 1.} Set $b_k$ as our candidate for $|P_\ell|$.
    
    \noindent \textbf{Step 2.} Let $i$ and $j$ be maximum such that the coefficient of $s^j t^i x^{b_k+i} y^{|P|-b_k+j} z^{i+j}$ has non-zero coefficient.  In particular, this coefficient will be $1$.  
    
    If $c_k = i \cdot j$, then the term $c_k stx^{b_k+1}y^{|P|-b_k+1}z^2$ corresponds to all the cover relations between maximal elements of $P_\ell$ and minimal elements of $P_{\ell+1}$, and the term $s^j t^i x^{b_k+i} y^{|P|-b_k+j} z^{i+j}$ corresponds to the upset $U$ generated by the maximal elements of $P_\ell$ and the downset $D$ generated by the minimal elements of $P_{\ell+1}$.  This is because for each element $w$ of $P$, either $u \leq w$ for some element $u$ of the $j$-antichain or $w \leq v$ for some element $v$ of the $i$-antichain, and the intersection $U \cap D$ consists only of the elements in these antichains.  Hence setting $Q$ to be the poset that is the downset of the $j$-antichain and $Q'$ to be the poset that is the upset of the $i$-antichain, we have $P = Q \oplus Q'$.  Thus $i$ is the number of minimal elements of $P_{\ell+1}$, $j$ the number of maximal elements of $P_\ell$, and $|P_\ell| = b_k - |P_1| - \cdots - |P_{\ell-1}|$, with $|P_0| = 0$.  Repeat \emph{Step 1} with the updated values $k = k+1$ and $\ell = \ell + 1$.  
    
    If $c_k \neq i \cdot j$ (in which case $ c_k < i \cdot j$), then we have not found the boundary between $P_\ell$ and $P_{\ell+1}$ at this $b_k$ so repeat \emph{Step 1} with the updated value $k = k+1$.
\end{proof}


\begin{lm}
    Let $P_1, \ldots, P_r$ and $P$ be as in Lemma \ref{lm:ord_sum_sizes}.  Then for each $1 \leq \ell \leq r$, we can recover $\Phi(P_\ell ; s, t, x, y, z)$ from $\Phi(P ; s, t, x, y, z)$.
    \label{lm:ord_sum_recover}
\end{lm}

\begin{proof}
    From Lemma \ref{lm:ord_sum_sizes}, we know the sequence of cardinalities $|P_1|, \ldots, |P_r|$.  A (non-constant) term $c_{U, D} s^{|A(U)|} t^{|A(D)|} x^{|D|} y^{|U|} z^{|U \cap D|}$ of $\Phi(P ; s, t, x, y, z)$ corresponds to a non-empty upset/downset pair $U \preceq D$ such that $A(U), A(D) \in P_\ell$ if and only if $|D| > |P_1| + \cdots + |P_{\ell-1}|$ and $|U| > |P_{\ell+1}| + \cdots + |P_r|$.  Thus each (non-constant) term of $\Phi(P_\ell ; s, t, x, y, z)$ is exactly of the form $c_{U, D} s^{|A(U)|} t^{|A(D)|} x^{|D|-h_\ell} y^{|U|-h_\ell'} z^{U \cap D}$ where $h_\ell = |P_1| + \cdots + |P_{\ell-1}|$ and $h_\ell' = |P_{\ell+1}| + \cdots + |P_r|$.
\end{proof}

\noindent The next lemma establishes that a generalized interval polynomial of an SP poset factors non-trivially if and only if the poset is a disjoint union.

\begin{lm}
    Let $P = P_1 \oplus P_2$ for any non-empty posets $P_1, P_2$.  Then $\Phi(P ; s, t, x, y, z)$ does not factor into two non-constant polynomials each with constant coefficient 1.
    \label{lm:ord_sum_factor}
\end{lm}

\begin{proof}
    Toward contradiction, suppose $\Phi(P ; s, t, x, y, z)$ does factor.  Then $\Phi(P ; s, t, x, y, z) = (1+F)(1+G)$ where $F, G$ are polynomials in $s, t, x, y, z$.  Since $\Phi(P ; s, t, x, y, z) = 1 + F + G + FG$, terms of $F$ and $G$ are also terms of $\Phi(P ; s, t, x, y, z)$.  This implies that each term of $F$ and $G$ has an exponent of at least one on each variable and has non-negative integer coefficient.  Again, we can determine $|P_1|$ and $|P_2|$ from $\Phi(P ; s, t, x, y, z)$.  We assume that $|P_1| \geq |P_2|$; The argument when $|P_1| \leq |P_2|$ is analogous.
    
    Observe that terms of $\Phi(P ; s, t, x, y, z)$ that have the form $c s^k t^k x^b y^a z^k$ come from an upset/downset pair $U \preceq D$ where $U$ and $D$ are generated by the same antichain of length $k$.  Let $w$ be the largest integer such that there is a term in $\Phi(P ; s, t, x, y, z)$ with non-zero coefficient such that the exponent on $t$ is $w$ and the exponent of $x$ is at least $|P_1|+w$.  Then $w = {\rm width}(P_2)$, since one of such terms corresponds to $U \preceq D$ where $|A(D)| = w$ and $|D| \geq |P_1|+w > |P_1|$, implying that $A(D)$ is an antichain in $P_2$.  Let $T$ be a term in $\Phi(P ; s, t, x, y, z)$ with non-zero coefficient corresponding to $U \preceq D$ where $U, D$ are generated by the same antichain of length $w$ in $P_2$.  So $T = c s^w t^w x^b y^a z^w$, where $c > 0, a \leq |P_2|, b \geq |P_1|+w$, and $a + b = |P|+w$.  Then $T$ is a term of $FG$, since if, without loss of generality, $T$ was a term of $F$, then the product of $T$ with any term of $G$ would have an exponent of $s$ greater than $w$, a contradiction to how we chose $w$.
    
    So, $T = T_F T_G$ where $T_F = c_F s^{w_{1, F}} t^{w_{2, F}} x^{b_F} y^{a_F} z^{w_{3, F}}$ and $T_G = c_G s^{w_{1, G}} t^{w_{2, G}} x^{b_G} y^{a_G} z^{w_{3, G}}$ are terms from $F$ and $G$ respectively.  Thus $b_F + b_G = b$, $a_F + a_G = a$, and $w = w_{1, F} + w_{1, G} = w_{2, F} + w_{2, G} = w_{3, F} + w_{3, G}$.  These imply $2w = w_{1, F} + w_{2, F} + w_{1, G} + w_{2, G}$, and also that $2w = 2w_{3, F} + 2w_{3, G}$.  Since $T_F, T_G$ are also terms of $\Phi(P ; s, t, x, y, z)$, we have $w_{3, F} \geq \max(w_{1, F}, w_{2, F})$ and $w_{3, G} \geq \max(w_{1, G}, w_{2, G})$.  This implies $2w_{3, F} \geq w_{1, F} + w_{2, F}$ and $2w_{3, G} \geq w_{1, G} + w_{2, G}$.  If either of these inequalities are strict, then $2w_{3, F} + 2w_{3, G} > w_{1, F} + w_{2, F} + w_{1, G} + w_{2, G}$, which is a contradiction since both sides of the inequality are $2w$ from above.  Thus $2w_{3, F} = w_{1, F} + w_{2, F}$ and $2w_{3, G} = w_{1, G} + w_{2, G}$, implying $\max(w_{1, F}, w_{2, F}) \leq (w_{1, F} + w_{2, F})/2$ and $\max(w_{1, G}, w_{2, G}) \leq (w_{1, G} + w_{2, G})/2$.  These inequalities hold if and only if $w_{1, F} = w_{2, F}$ and $w_{1, G} = w_{2, G}$.  Hence we conclude $w_{1, F} = w_{2, F} = w_{3, F}$ and $w_{1, G} = w_{2, G} = w_{3, G}$.
    
    So $T_F$ has the form $c_F s^{w_F} t^{w_F} x^{b_F} y^{a_F} z^{w_F}$ and $T_G$ has the form $c_G s^{w_G} t^{w_G} x^{b_G} y^{a_G} z^{w_G}$, implying that they correspond to upset/downset pairs generated by antichains $A_F, A_G$ of length $w_F, w_G$, respectively.  If either of $A_F$ or $A_G$ is in $P_1$, then $a_F > |P_2|$ or $a_G > |P_2|$ and hence $a > |P_2|$, a contradiction.  But if both $A_F$ and $A_G$ are in $P_2$, then $b_F > |P_1|$ and $b_G > |P_1|$ implying $b > 2|P_1| \geq |P|$, also a contradiction.  Therefore, $\Phi(P ; s, t, x, y, z)$ cannot factor.
\end{proof}

Now we have enough information to inductively construct the composition tree of an SP poset from its generalized interval polynomial.

\begin{tm}
    If $P$ is an SP poset, then $P$ can be uniquely reconstructed from $\Phi(P ; s, t, x, y, z)$.
\end{tm}

\begin{proof}
    We proceed by induction on $|P|$.  If $|P| = 1$, then there is only one choice of $P$ and so the result trivially holds.  Suppose the result holds for SP posets with fewer than $|P|$ elements.  Let $\Phi(P ; s, t, x, y, z) = F_1 F_2 \cdots F_r$ be the unique minimal factorization of $\Phi(P ; s, t, x, y, z)$ scaled so all factors have constant coefficient 1.  From Lemma \ref{lm:ord_sum_factor}, we know that the polynomials of ordinal sums do not factor; it follows that $P$ has exactly $r$ components $P_1, \ldots, P_r$, i.e.\ $P = P_1 + \cdots + P_r$. From Proposition \ref{prop:union}, we have $F_i = \Phi(P_i ; s, t, x, y, z)$ (after some labelling of the $P_1, \ldots, P_r$).  For each $i$, since $P_i$ is not a disjoint union of smaller, non-empty posets, it is an ordinal sum.  So $P_i = Q_{i, 1} \oplus \cdots \oplus Q_{i, \ell}$, where each $Q_{i, j}$, $1 \leq j \leq \ell$, is not an ordinal sum of smaller, non-empty posets.  By Lemma \ref{lm:ord_sum_recover}, we can recover $\Phi(Q_{i, j} ; s, t, x, y, z)$ for each $j$ from $\Phi(P ; s, t, x, y, z)$.  By induction, we can uniquely determine each $Q_{i, j}$, and hence $P$.
\end{proof}

This gives a positive answer to the question of distinguishing SP posets from their generalized interval polynomials.

\begin{cor}
    Given SP posets $P$ and $Q$, if $\Phi(P; s, t, x, y, z) = \Phi(Q; s, t, x, y, z)$, then $P \cong Q$.
\end{cor}

\section{Distributions in Causal Sets}

\subsection{Causal Set Theory}

Causal Set Theory (CST) is a theory of quantum gravity in which spacetime is treated as a locally finite poset, or a \emph{causal set}.  Elements of the causal set are points in spacetime, and the relations in the causal set represent causal relations of the spacetime.  A thorough review of the development of CST can be found in \cite{surya2019causal}.  Given a causal set, one would like to extract physically relevant information directly from the partial order, such as Minkowski dimension, geodesics, or a d'Alembertian.

One way of constructing causal sets is to ``sprinkle'' points into a spacetime manifold according to some random Poisson process \cite{surya2019causal}.  This procedure produces manifold-like causal sets, but since it assumes a pre-existing spacetime it is not a viable process for how a causal set would actually be formed in nature.  However, sprinklings can be useful for verifying that discrete properties are physically meaningful in a continuum limit.

\subsection{Estimating Dimension}

In this section we focus on the problem of determining the Minkowski dimension of a causal set.  There are numerous methods for estimating this dimension (see, for example, \cite{surya2019causal}, \cite{aghili2019discrete}).  Here, we will focus on one method involving the distribution of interval sizes.

In \cite{glaser2013towards}, the authors determined the asymptotic behaviour of the distribution of interval sizes of a causal set sprinkled into a $d$-dimensional causal diamond.  Let $N_k^d$ be the number of intervals of cardinality $k+2$ in an $n$-element causal set sprinkled into a $d$-dimensional causal diamond.  The authors determined that for $k \geq 0$,

\[ s_k^d = \lim_{n \to \infty} \frac{N_k^d}{N_0^d} = \frac{\Gamma(\frac{2}{d}+k)}{\Gamma(\frac{2}{k}) \Gamma(k+1)} = \binom{\frac{2}{d}+k-1}{k}, \quad {\rm and \; so} \quad S_d(z) = \sum_{k \geq 0} s_k^d z^k = \frac{1}{(1-z)^{\frac{2}{d}}} \]

\noindent Consequently, the generating series of interval sizes contains information about the Minkowski dimension of a causal set.  Note that $[z]S_d(z) = s_1^d = \frac{2}{d}$, which is the limit of the proportion of $3$-element intervals divided by the number of $2$-element intervals.

One shortcoming of estimating dimension through the distribution of interval sizes is that the convergence of the values $s_k^d$ is slow \cite{aghili2019discrete}.  This led us to the following question: can the extra information from the distribution of \emph{generalized interval} sizes offer a better estimate for causal set dimension?

\acknowledgements{Thanks to Rafael Sorkin and Sumati Surya for useful conversations.}


\bibliographystyle{plain}
\bibliography{refs}

\end{document}